\documentclass[10pt]{article}

\usepackage[margin=1.05in]{geometry}
\usepackage{amsmath,amssymb,amsthm,mathtools}
\usepackage{authblk}
\usepackage{enumitem}
\usepackage{booktabs}
\usepackage{longtable}
\usepackage{array}
\usepackage{hyperref}
\hypersetup{colorlinks=true,linkcolor=blue,citecolor=blue,urlcolor=blue}
\setlist{topsep=0.35em,itemsep=0.15em,parsep=0pt}

\newtheorem{theorem}{Theorem}[section]
\newtheorem{lemma}[theorem]{Lemma}

\theoremstyle{definition}
\newtheorem{definition}[theorem]{Definition}
\newtheorem{remark}[theorem]{Remark}

\newcommand{\calF}{\mathcal F}
\newcommand{\calU}{\mathcal U}

\newcommand{\calI}{\mathcal I}
\newcommand{\calA}{\mathcal A}

\newcommand{\calH}{\mathcal H}
\newcommand{\calW}{\mathcal W}
\newcommand{\calS}{\mathcal S}

\newcommand{\agr}{\operatorname{agr}}

\begin{document}

\title{{\Large\bf  Non-trivial Intersection Problems for Multi-partite Hypergraphs\thanks{Research  supported by National Key R$\&$D Program of China (Grant No. 2023YFA1010202).}}}
\date{\today}
\author[1]{Jianfeng Hou\thanks{ Email: \texttt{jfhou@fzu.edu.cn}}}
\author[1]{Caiyun Hu \thanks{Email: \texttt{hucaiyun.fzu@gmail.com}}}  
\affil[1]{Center for Discrete Mathematics,
            Fuzhou University, Fujian, 350003, China}
\maketitle

\begin{abstract}

We study non-trivial intersection problems for multi-partite hypergraphs, excluding the usual extremal examples determined by fixed vertices or fixed coordinates.  Our first result determines the exact value of the non-trivial $t$-intersection problem in the symmetric product $[n]^r$ for $1\le t\le r-2$ and all $n\ge2$.  Frankl and Nie proved a two-candidate formula for sufficiently large $n$ and conjectured it for all $n\ge 2$; our formula shows that the conjectured expression must be enlarged, in small ranges of $n$, by additional ball-type terms arising from the Frankl families.

Our second result concerns intersecting families in general products $X_1\times\cdots\times X_r$, where $|X_i|=n_i$, with no common vertex.  Let $m_0(1,n_1,\ldots,n_r)$ denote the largest size of such a family.  We show that this number is equal to the maximum of $\sum_{X\in \mathcal{D}}\prod_{i\in X}(n_i-1)$ over all downsets $\mathcal{D}\subseteq 2^{[r]}$ such that $\bigcup_{X\in \mathcal{D}}X=[r]$ and no two members of $\mathcal{D}$ have union $[r]$.  This finite reduction separates the intersection obstruction from the part sizes and yields explicit fully asymmetric formulas for $r=4,5,6$.
\end{abstract}

\textbf{Keywords:} Non-trivial, intersection, multi-partite hypergraph, shifting.

\section{Introduction}

Throughout $[r]=\{1,\ldots,r\}$.  Let $r\ge2$ and let $X_1,\ldots,X_r$ be finite sets.   An \emph{$r$-partite $r$-graph} with parts $X_1,\ldots,X_r$ will be identified with a family
\[
   \calF\subseteq X_1\times\cdots\times X_r
   =\{(x_1,\ldots,x_r):x_i\in X_i,\ 1\le i\le r\}.
\]
The \emph{fixed-coordinate set} of $\mathcal{F}$ is defined as 
\[
   \bigcap \calF
   :=
   \{i\in [r]: \text{all sequences in } \calF \text{ have the same entry in the } i\text{-th coordinate}\}.
\]
For two vectors $\mathbf{x}=(x_1,\ldots,x_r)$ and $\mathbf{y}=(y_1,\ldots,y_r)$ in the same product, let 
\[
   \agr(\mathbf{x},\mathbf{y})
   :=
   \{i\in[r]:x_i=y_i\}.
\]
Thus $\mathbf{x}$ and $\mathbf{y}$ are disjoint as hypergraph edges if and only if $\agr(\mathbf{x},\mathbf{y})=\emptyset$.  A family $\calF\subseteq[n]^r$ is \emph{$t$-intersecting} if $|\agr(\mathbf{x},\mathbf{y})|\ge t$ for all $\mathbf{x},\mathbf{y}\in\calF$; when $t=1$ we simply say intersecting.  A $t$-intersecting family $\calF$ is \emph{non-trivial} if $|\bigcap\calF|<t$.



\subsection{\texorpdfstring{Non-trivial intersection in $r$-partite $r$-graphs}{Non-trivial intersection in r-partite r-graphs}}

The classical Erd\H{o}s--Ko--Rado theorem~\cite{EKR1961} states that if $\mathcal{F}\subseteq\binom{[n]}{r}$ is intersecting and $n\ge 2r$, then $|\mathcal{F}|\le\binom{n-1}{r-1}$. The extremal families are the trivial stars, consisting of all $r$-sets containing a fixed vertex. This naturally leads to the corresponding non-trivial problem, in which such stars are excluded. Hilton and Milner~\cite{HiltonMilner1967} resolved this question by showing that any non-trivial intersecting family $\mathcal{F}\subseteq\binom{[n]}{r}$ with $n>2r$ satisfies
\[
|\mathcal{F}|
   \le
   \binom{n-1}{r-1}-\binom{n-r-1}{r-1}+1.
\]
The bound is sharp: The extremal family is called the Hilton--Milner
family
\[
   \mathcal H_{x,B}(n)
   =
   \{B\}
   \cup
   \left\{
      F\in\binom{[n]}{r}:
      x\in F,\ F\cap B\neq\emptyset
   \right\},~\text{ where }~ x\in[n] ~\text{and}~B\in\binom{[n]\setminus\{x\}}{r}.
\]

The Hilton--Milner theorem, whose short proof was later given by Frankl and F\"uredi~\cite{FranklFuredi1986},  initiated a line of work in which the common core of an intersecting family is controlled or excluded.  For $t>1$, the corresponding non-trivial problem asks for the largest $t$-intersecting family. 
Frankl \cite{FranklIntersecting1978} determined this maximum for sufficiently large \(n\). The extremal family is
\[
   \mathcal A_1(n,r,t)
   =
   \left\{
      F\in \binom{[n]}{r}: |F\cap [t+2]|\ge t+1
   \right\}
\]
when \(t+1\le r\le 2t+1\), while for \(r>2t+1\) it is the
Hilton--Milner type family
\[
   \mathcal H_{r,t}(n)
   = \left\{
      [r+1]\setminus\{i\}: i\in [r]
   \right\} \cup 
   \left\{
      F\in \binom{[n]}{r}: [t]\subseteq F,\,
      F\cap [t+1,r+1]\neq \emptyset
   \right\},
\]
where $[t+1, r+1] = \{t+1, t+2, \dots, r+1\}$. 
More generally, Frankl \cite{Frankl1978} introduced the families
\[
   \mathcal A_i(n,r,t)
   =
   \left\{
      F\in \binom{[n]}{r}: |F\cap [t+2i]|\ge t+i
   \right\},~1\le i\le \left\lfloor\frac{r-t}{2}\right\rfloor.
\]
and proposed that these families govern the complete t-intersection problem.
Ahlswede and Khachatrian \cite{AK1996} later proved the complete non-trivial intersection theorem, determining the maximum for all relevant $n,r,t$.
Thus the extremal construction is not merely the Hilton–Milner construction with $t$ fixed coordinates; depending on the parameter range, ball-type Frankl families also appear.

Turning to multi-partite versions, Kwan, Sudakov, and Vieira~\cite{KwanSudakovVieira2018} studied non-trivially intersecting multi-partite families, where the ground set is partitioned into parts and each member has a prescribed intersection size with each part. The extremal examples are multi-partite Hilton--Milner-type families, constructed by choosing a distinguished part together with possible auxiliary parts.

Our first problem concerns $t$-intersecting families in $r$-partite $r$-graphs. 
For positive integers $n,t, r$ with $n\ge r>t$, define
\[
   \iota_0(t,n;r):=
   \max\{|\calF|:\calF\subseteq[n]^r\text{ is non-trivial and }t\text{-intersecting}\}.
\]
Frankl and Nie \cite{FranklNie2026} determined this maximum for sufficiently large $n$, proving that
\[
    \iota_0(t,n;r)
    =
    \max\left\{
        n^{r-t}-(n-1)^{r-t}+t(n-1),
        (t+2)n^{r-t-1}-(t+1)n^{r-t-2}
    \right\},
\]
and conjectured that the same two-candidate formula holds for every $n\ge2$.  We show that this conjecture is asymptotically correct but fails for certain values of $n$: in small parameter ranges, additional ball-type examples arising from the Frankl families may achieve larger size.

\begin{theorem}\label{thm:iota-symmetric}
For integers $r\ge3$, $1\le t\le r-2$, $n\ge 2$ and $q:=r-t$, we have 
\[
   \iota_0(t,n;r)
   =
   \max\left\{
      W_{r,t}(n),\,
      A_i(t,n;r):1\le i\le\left\lfloor\frac q2\right\rfloor
   \right\},
\]
where
\[
   W_{r,t}(n):=n^q-(n-1)^q+t(n-1),
\]
and
\[
   A_i(t,n;r):=
   n^{q-2i}\sum_{j=0}^i\binom{t+2i}{j}(n-1)^j.
\]
\end{theorem}

In Theorem \ref{thm:iota-symmetric}, the term $W_{r,t}(n)$ corresponds to a multi-partite Hilton--Milner construction, while $A_i(t,n;r)$ is the product analogue of the $i$-th Frankl family
For fixed $r$ and $t$,  two terms $W_{r,t}(n)$ and $A_1(t,n;r)$ have the largest possible degree, namely $r-t-1$, whereas $A_i(t,n;r)$ has degree at most $r-t-2$ for every $i\ge2$.  Hence the formula reduces to the Frankl--Nie two-candidate maximum when $n$ is sufficiently large, but for small $n$ the lower-degree ball terms may dominate.

\subsection{Hilton--Milner case in general multi-partite products}

We now turn to a different non-trivial condition, formulated in terms of matching and covering numbers.  For a family $\mathcal F\subseteq X_1\times\cdots\times X_r$, let $\nu(\mathcal F)$ be the maximum number of pairwise disjoint edges in $\mathcal F$, and let $\tau(\mathcal F)$ be the minimum size of a vertex set meeting every edge.  For $r=2$, K\"{o}nig's theorem \cite{DG1931} gives $\nu(\mathcal F)=\tau(\mathcal F)$; for $r\ge3$ the two parameters differ in an essential way.

Our second result deals with the Hilton--Milner case in general multi-partite products.  Given $\mathcal F\subseteq X_1\times\cdots\times X_r$, $\mathcal F$ is intersecting means $\nu(\mathcal F)\le1$.  Meyer \cite{Meyer1974} proved that every intersecting subhypergraph of $[n]^r$ has at most $n^{r-1}$ edges.  
Deza and Frankl \cite{DezaFrankl1983} proved the asymmetric version: if $n_1\ge n_2\ge\cdots\ge n_r$, then every intersecting subhypergraph of $[n_1]\times\cdots\times[n_r]$ has at most $\prod_{\ell=1}^{r-1}n_\ell$ edges.  
More generally, Frankl \cite{Frankl2012} showed that $\nu(\mathcal F)\le s$ implies $|\mathcal F|\le s\prod_{\ell=1}^{r-1}n_\ell$ in the same ordered setting.  Since these extremal examples have small covering number, the natural non-trivial condition is $\tau(\mathcal F)>s$; here we treat the case $s=1$.

Following Frankl and Nie \cite{FranklNie2026}, for $n_i=|X_i|$ we define $m_0(1,n_1,\ldots,n_r)$ to be the maximum size of a family $\mathcal F\subseteq X_1\times\cdots\times X_r$ satisfying
\[
   \nu(\mathcal F)\le1<\tau(\mathcal F).
\]
Equivalently, $\mathcal F$ is intersecting but has no common vertex.  This agrees with $\iota_0(1,n;r)$ in the symmetric case $n_1=\cdots=n_r=n$, although for $t>1$ the condition $|\bigcap\mathcal F|<t$ differs from the covering condition $\tau(\mathcal F)>t$.
Frankl and Nie \cite{FranklNie2026} studied this problem and proved, among other results, the symmetric formula
\[
    m_0(1,n,\ldots,n)=n^{r-1}-(n-1)^{r-1}+n-1
\]
for all $r\ge 3$ and $n\ge 2$, and the asymmetric three-part formula
\[
    m_0(1,n_1,n_2,n_3)=n_1+n_2+n_3-2.
\]
The symmetric formula above is the case $s=1$ of a conjecture of Lu and Ma \cite{LuMa2024}, which states that, for sufficiently large $n$, $r\ge 4$ and $s<n$,
\[
m_0(s,n,\ldots,n)=s n^{r-1}-(n-1)^{r-1}+n-s,
\]
where $m_0(s,n,\ldots,n)$ denotes the analogous maximum under $\nu(F)\le s<\tau(F)$.
Frankl and Nie \cite{FranklNie2026} later confirmed this conjecture for sufficiently large $n$, and also pointed out that, already for four parts, no single asymmetric expression governs all parameter ranges. We give a finite exact reduction which accounts for this dependence on the part sizes.

Let $a_i:=n_i-1$ for $i\in[r]$. For a subset $X\subseteq[r]$, set 
\[
   a_X:=\prod_{i\in X}a_i,  \text{ and } 
   \qquad a_\emptyset:=1.
\]

A family $\mathcal{D}\subseteq 2^{[r]}$ is called \emph{available} if it satisfies the following  
\begin{enumerate}[label=(D\arabic*)]
   \item $\mathcal{D}$ is a downset: if $X\in \mathcal{D}$ and $Y\subseteq X$, then $Y\in \mathcal{D}$;
   \item $\bigcup_{X\in \mathcal{D}}X=[r]$;
   \item $X\cup Y\ne[r]$ for all $X,Y\in \mathcal{D}$.
\end{enumerate}
Let $\mathfrak D_r$ be the collection of all available families  $\mathcal{D}\subseteq 2^{[r]}$.

\begin{theorem}
\label{thm:downset}
For every $r\ge3$ and $n_1,\ldots,n_r\ge2$,
\[
   m_0(1,n_1,\ldots,n_r)
   =
   \max_{\mathcal{D}\in\mathfrak D_r}\sum_{X\in \mathcal{D}}a_X.
\]
\end{theorem}

Theorem~\ref{thm:downset} reduces the multi-partite Hilton--Milner problem to a weighted problem on the Boolean lattice.  Applying this reduction, we determine the fully asymmetric values for $r=4,5,6$.  For $r=4$, the optimization reduces to intersecting graphs on four vertices and is governed by the star--triangle dichotomy.  For $r=5$, weighted intersecting graphs still suffice.  For $r=6$, one must also choose among complementary pairs of triples.  These cases illustrate the increasing complexity of the asymmetric problem.

\subsection{Proof ideas and organization}

The proofs of the two main results use a common shifting philosophy, but in different ways.  In both problems we first replace an extremal family by a shifted one without changing its size or the relevant intersection property.  This makes it possible to record an edge by the set of coordinates in which it differs from a fixed reference edge, and hence to pass from the multi-partite product to a weighted problem on $2^{[r]}$.

For the symmetric $t$-intersection problem, this reduction leads to a weighted non-trivial $t$-intersection problem on the Boolean lattice.  The weights are determined by the product structure: a coordinate in the support contributes a factor $n-1$, while the remaining coordinates contribute freely.  
After this translation, 
the complete non-trivial intersection theorem of Ahlswede and Khachatrian, together with the weighted theorem of Bey and Engel, identifies the relevant extremal shapes: a Hilton–Milner-type family and the Frankl-family balls. Evaluating their weights gives the term 
$W_{r,t}(n)$ and the ball terms $A_i(t,n;r)$ in Theorem~\ref{thm:iota-symmetric}.

For the multi-partite Hilton--Milner problem, the shifted family is encoded more directly by a downset $\mathcal{D}\subseteq2^{[r]}$.  The condition that the family has no common vertex becomes $\bigcup_{X\in \mathcal{D}}X=[r]$, while the intersecting condition becomes the finite obstruction $X\cup Y\ne[r]$ for all $X,Y\in \mathcal{D}$.  Thus the original product problem is converted into the weighted downset optimization in Theorem~\ref{thm:downset}.  The explicit formulas for $r=4,5,6$ are then obtained by solving this finite optimization: graphs suffice for $r=4$ and $r=5$, while for $r=6$ complementary pairs of triples also have to be considered.

The paper is organized as follows.  Section~$2$ contains the shifting and weighted intersection tools.  Section~$3$ proves Theorem~\ref{thm:iota-symmetric}.  Section~$4$ proves Theorem~\ref{thm:downset} and derives the explicit formulas for $r=4,5,6$.

\section{Preliminaries}

In this section, we collect the notation and preliminary results needed for our proofs. We begin by introducing a multi-partite adaptation of the classical shifting operation, due to Kwan, Sudakov, and Vieira~\cite{KwanSudakovVieira2018}.


\begin{definition}(Shifting)
Let $\mathcal{F} \subseteq X_1 \times \cdots \times X_r$ be an $r$-partite $r$-graph, and let $\ell, j$ be integers with  $1\le \ell\le r$ and  $1<j\le |X_\ell|$. 
Define the \emph{$(1\leftarrow j)$-shift under $[n_{\ell}]$} by $S_j^{(\ell)}(\mathcal{F})=\{S_j^{(\ell)}(F):F=(a_1,\ldots,a_\ell,\ldots,a_r)\in\mathcal{F}\}$, where
\[
S_j^{(\ell)}\bigl((a_1,\ldots,a_\ell,\ldots,a_r)\bigr)=
\begin{cases}
F'=(a_1,\ldots,1,\ldots,a_r) & \text{if } a_\ell=j \text{ and } F'\notin\mathcal{F},\\
(a_1,\ldots,a_r) & \text{otherwise.}
\end{cases}
\]
\end{definition}
We say that $\mathcal{F}$ is \emph{$\ell$-shifted} if $S_j^{(\ell)}(\mathcal{F})=\mathcal{F}$ for all $j$. 
Further, call $\mathcal{F}$ \emph{coordinate-wise shifted} if it is $\ell$-shifted for every $1\le \ell\le r$. The following result, given by Kwan, Sudakov and Vieira~\cite{KwanSudakovVieira2018}, shows that shifting operation maintains the intersecting property.

\begin{theorem}[Kwan--Sudakov--Vieira~\cite{KwanSudakovVieira2018}]
\label{thm:ksv-shifted}
Let $r\ge 2$ and $n_1,\ldots,n_r\ge 2$. If there exists a non-trivial 
intersecting family $\mathcal{F} \subseteq [n_1] \times \cdots \times [n_r]$, then there exists a maximum-size non-trivial intersecting family  $\mathcal{F}' \subseteq [n_1] \times \cdots \times [n_r]$ which is coordinate-wise shifted.
\end{theorem}

We remark that the original definition and theorem of Kwan--Sudakov--Vieira are more general than the form stated above.  It applies to families of separated sets
\[
\mathcal{G}\subseteq \prod_{s=1}^p \binom{X_s}{k_s}
\]
for arbitrary positive integers $p,k_1,\ldots,k_p$, and uses the usual shifts inside each part. For our purposes, the simplest case stated above (i.e., $p = r$ and $k_1 = \cdots = k_r = 1$) suffices.


For a vector
$\mathbf{x}=(x_1,\ldots,x_r)$, define its \emph{projection} as 
\[
   P(\mathbf{x}):=\{i\in[r]:x_i=1\}.
\]
For a shifted family of tuples, projections reduce the problem to set systems.  

\begin{lemma}[Frankl--Nie~\cite{FranklNie2026}]
\label{lem:FN-projection}
For integers $r\ge 3$ and $1\le t\le r-2$, if $\mathcal{F} \subseteq [n_1] \times \cdots \times [n_r]$ is coordinate-wise shifted and non-trivial
$t$-intersecting, then the family of projections
\[
   P(\calF):=\{P(\mathbf{x}):\mathbf{x}\in\calF\}\subseteq 2^{[r]}
\]
is non-trivial $t$-intersecting as well.
\end{lemma}

A major problem with the shifting operation is that it may decrease the covering number $\tau$; in particular, it might turn a non-trivial family into a trivial one. 
We say a non-trivial $t$-intersecting family $\mathcal{F}\subseteq X_1\times\dots \times X_r$ is \emph{$\ell$-shift-resistant} 
if there exists $x \in X_{\ell}$ such that $S^{(\ell)}_x (\mathcal{F})$ is trivial, i.e. $| \cap S^{(\ell)}_x (\mathcal{F})| = t$. 
Frankl and Nie also proved the following shift-resistant bound, which is used only in the symmetric
$t$-intersection theorem below.

\begin{lemma}[Frankl--Nie \cite{FranklNie2026}]
\label{lem:FN-resistant}
For integers  $n\ge 2$, $r\ge 3$ and $r-2\ge t\ge 1$, let  $\mathcal{F}\subseteq [n]^r$ be non-trivial
$t$-intersecting such that for every $1\le \ell\le r$, $\mathcal{F}$ is either $\ell$-shifted or
$\ell$-shift-resistant, and $\mathcal{F}$ is not
coordinate-wise shifted. Then
\[
|\mathcal{F}|\le (t+2)n^{r-t-1}-(t+1)n^{r-t-2}.
\]
\end{lemma}

In order to prove Theorem \ref{thm:iota-symmetric}, we shall use the following \(p\)-biased consequence of the weighted non-trivial \(t\)-intersection theorem of Bey and Engel \cite{BeyEngel2000}. We state it in a form adapted to the product measure used in this paper.
For \(0<p\le 1\) and \(\mathcal U\subseteq 2^{[r]}\), define the
\emph{\(p\)-biased measure} of \(\mathcal U\) by
\[
    \mu_p(\mathcal U)
    =
    \sum_{A\in\mathcal U} p^{|A|}(1-p)^{r-|A|}.
\]

\begin{theorem}[Bey--Engel~\cite{BeyEngel2000}]
\label{thm:BE}
Let $0<p\le 1/2$ and $1\le t\le r-2$.  Let $\calU\subseteq 2^{[r]}$ be non-trivial and $t$-intersecting, meaning that
\[
   |A\cap B|\ge t\quad\text{for all }A,B\in\calU,
   \qquad
   \left|\bigcap_{A\in\calU}A\right|<t.
\]
For $1\le i\le \lfloor (r-t)/2\rfloor$, define the Frankl families
\[
   \calS_i:=\{A\subseteq[r]: |A\cap[t+2i]|\ge t+i\}.
\]
Define the Hilton--Milner type family
\[
   \calH_{r,t}:=
   \{A\subseteq[r]:[t]\subseteq A,\, A\cap\{t+1,\ldots,r\}\ne\emptyset\}
   \cup
   \{[r]\setminus\{h\}:h\in[t]\}.
\]
Then
\[
   \mu_p(\calU)
   \le
   \max\left\{
      \mu_p(\calH_{r,t}),\,
      \mu_p(\mathcal{S}_i):1\le i\le \left\lfloor\frac{r-t}{2}\right\rfloor
   \right\}.
\]
\end{theorem}

\begin{remark} Theorem~\ref{thm:BE} is the \(p\)-biased specialization of a
more general size-dependent weighted theorem of Bey and Engel.  In that
setting one fixes non-negative weights \(w_0,\ldots,w_r\) and considers
\[
    W(\mathcal V)
    :=
    \sum_{A\in\mathcal V} w_{|A|}
    \qquad
    (\mathcal V\subseteq 2^{[r]}).
\]
Bey--Engel's non-trivial \(t\)-intersection theorem gives the corresponding maximum, under their hypotheses, in terms of their weighted Frankl and Hilton--Milner type candidates.
The form used here is obtained by taking the weight
\[
    w_j=\alpha^{-j},
    \qquad
    \alpha=\frac{1-p}{p}.
\]
Since \(0<p\le 1/2\), we have \(\alpha\ge 1\), and for every
\(\mathcal V\subseteq 2^{[r]}\),
\[
    \mu_p(\mathcal V)
    =
    \sum_{A\in\mathcal V}p^{|A|}(1-p)^{r-|A|}
    =
    (1-p)^r\sum_{A\in\mathcal V}\alpha^{-|A|}
    =
    (1-p)^r W(\mathcal V).
\]
The factor \((1-p)^r\) is independent of \(\mathcal V\).  Hence
maximizing \(p\)-biased measure is equivalent to maximizing the
Bey--Engel size-dependent weight with \(w_j=\alpha^{-j}\).  After this
normalization and the harmless relabelling of the ground-set size as
\(r\), the weighted theorem gives exactly the form stated in
Theorem~\ref{thm:BE}.  
\end{remark}

\section[Symmetric all-n formula]{The symmetric all-$n$ formula for $\iota_0(t,n;r)$}

In this section, we prove Theorem \ref{thm:iota-symmetric} using Theorem \ref{thm:BE}. Let  $r\ge3$, $t\le r-2$ and $n\ge 2$ be positive integers and let $q=r-t$.  We first prove the lower bound.  Write $T=[t]$ and $Q=\{t+1,\ldots,r\}$.  Let
\[
   \calW_1:=\{\mathbf{x}\in[n]^r:x_i=1\text{ for all }i\in T,\text{ and }x_s=1\text{ for some }s\in Q\},
\]
and
\[
   \calW_2:=
   \bigcup_{h\in T}\{\mathbf{x}\in[n]^r:x_i=1\text{ for all }i\ne h,\ x_h\ne1\}.
\]
Set $\calW:=\calW_1\cup\calW_2$.  Clearly $\left|\bigcap \calW\right|<t$ and 
\[
   |\calW|=n^q-(n-1)^q+t(n-1)=W_{r,t}(n).
\]
Now we show $\calW$ is $t$-intersecting. Indeed, if two vectors both lie in $\calW_1$, they agree on all $t$
coordinates of $T$.  If one lies in $\calW_1$ and the other in $\calW_2$, they agree on at least
$t-1$ coordinates of $T$ and also on one coordinate of $Q$ where the vector in $\calW_1$ has value
$1$.  If both lie in $\calW_2$, they differ from the all-$1$ vector in at most two coordinates, and
therefore agree in at least $r-2\ge t$ coordinates.  Thus $\calW$ is a non-trivial $t$-intersecting family, and then 
\[
 \iota_0(t,n;r)\ge W_{r,t}(n). 
\]

For $1\le i\le\lfloor q/2\rfloor$, define the following product-space Frankl family, equivalently a Hamming-ball-type family,
\[
   \calA_i:=
   \left\{\mathbf{x}\in[n]^r:
   \left|\{s\in[t+2i]:x_s\ne1\}\right|\le i
   \right\}.
\]
Any two members of $\calA_i$ have at least
\[
   (t+2i)-i-i=t
\]
coordinates among $[t+2i]$ where both are equal to $1$.  Hence $\calA_i$ is $t$-intersecting.  Since
$i\ge1$, no coordinate is fixed in all members of $\calA_i$, so it is non-trivial.  Counting by the
number $j$ of non-$1$ coordinates inside $[t+2i]$ gives
\[
   |\calA_i|=n^{q-2i}\sum_{j=0}^i\binom{t+2i}{j}(n-1)^j=A_i(t,n;r).
\]
Thus the lower bound follows.

For the upper bound, let $\calF\subseteq[n]^r$ be non-trivial and $t$-intersecting.  
We iteratively apply coordinate-wise shifts that do not destroy non-triviality until no further such shift is possible. 
This process terminates because each non-trivial shift strictly decreases the sum of all coordinate values over the family.
Since each shift preserves cardinality and $t$-intersection, the terminal family
$\calF^\ast$ satisfies
\[
   |\calF^\ast|=|\calF|,
   \qquad
   \left|\bigcap \calF^\ast\right|<t.
\]
For each coordinate, $\calF^\ast$ is either shifted in that coordinate or is shift-resistant in the sense of Lemma \ref{lem:FN-resistant}.  
If $\calF^\ast$ is not coordinate-wise shifted, Lemma \ref{lem:FN-resistant} gives
\[
   |\calF|
   =|\calF^\ast|
   \le (t+2)n^{q-1}-(t+1)n^{q-2}
   =A_1(t,n;r),
\]
which is already at most the claimed maximum.  

Otherwise, $\calF^\ast$ is coordinate-wise shifted.
Hence, we may assume without loss of generality that $\calF$ is coordinate-wise shifted.

Let
\[
   \calU:=P(\calF)=\{P(\mathbf{x}):\mathbf{x}\in\calF\}\subseteq2^{[r]}.
\]
By Lemma \ref{lem:FN-projection}, $\calU$ is non-trivial and $t$-intersecting.  For each
$A\in\calU$, there are exactly $(n-1)^{r-|A|}$ vectors in $[n]^r$ having projection $A$, and then 
\[
   |\calF|\le\sum_{A\in\calU}(n-1)^{r-|A|}.
\]
Put $p=1/n$.  Since $n\ge2$, we have $0<p\le1/2$. It follows from $(n-1)^{r-|A|}=n^r p^{|A|}(1-p)^{r-|A|}$ that 
\[
   |\calF|\le n^r\mu_p(\calU).
\]
By Theorem \ref{thm:BE},
\[
   \mu_p(\calU)
   \le
   \max\left\{
      \mu_p(\calH_{r,t}),\,
      \mu_p(\calS_i):1\le i\le\left\lfloor\frac q2\right\rfloor
   \right\}.
\]
Thus, to complete the proof, it remains only to compute these measures.

For $\calS_i=\{A:|A\cap[t+2i]|\ge t+i\}$, write
$j=|[t+2i]\setminus A|$.  Then $0\le j\le i$ and
\[
   \mu_p(\calS_i)=
   \sum_{j=0}^i\binom{t+2i}{j}p^{t+2i-j}(1-p)^j.
\]
Multiplying by $n^r$ and using $p=1/n$ gives
\[
   n^r\mu_p(\calS_i)
   =n^{q-2i}\sum_{j=0}^i\binom{t+2i}{j}(n-1)^j
   =A_i(t,n;r).
\]

For $\calH_{r,t}$, its first part
\[
   \{A:[t]\subseteq A,\ A\cap Q\ne\emptyset\},\qquad Q=\{t+1,\ldots,r\},
\]
has measure
\[
   p^t\bigl(1-(1-p)^q\bigr),
\]
so after multiplying by $n^r$ contributes
\[
   n^q-(n-1)^q.
\]
The second part consists of the $t$ sets $[r]\setminus\{h\}$ with $h\in[t]$.  Each has measure
$p^{r-1}(1-p)$, so after multiplying by $n^r$ contributes $n-1$.  Thus
\[
   n^r\mu_p(\calH_{r,t})=n^q-(n-1)^q+t(n-1)=W_{r,t}(n).
\]
Combining these estimates gives
\[
   |\calF|
   \le
   \max\left\{W_{r,t}(n),A_i(t,n;r):1\le i\le\left\lfloor\frac q2\right\rfloor\right\}.
\]
Together with the lower bound, this proves the theorem.


\section[Downset reduction]{A general downset reduction for $m_0(1,n_1,\ldots,n_r)$}

\begin{proof}[Proof of Theorem \ref{thm:downset}]
By Theorem \ref{thm:ksv-shifted}, there is a maximum non-trivial intersecting family $\calF\subseteq[n_1]\times\cdots\times[n_r]$
that is coordinate-wise shifted. Define
\[
   \calU:=P(\calF)=\{P(\mathbf{x}):\mathbf{x}\in\calF\}\subseteq2^{[r]}.
\]
Then $\calU$ is intersecting and $\bigcap\calU=\emptyset$ by Lemma \ref{lem:FN-projection}.
We next prove projection-saturation.  Suppose $A\in\calU$ and $\mathbf{x}$ is any vector with $P(\mathbf{x})=A$.
If $\mathbf{x}\notin\calF$, then adding $\mathbf{x}$ to $\calF$ preserves intersection: for every $\mathbf{y}\in\calF$, the sets
$A=P(\mathbf{x})$ and $P(\mathbf{y})$ intersect, so $\mathbf{x}$ and $\mathbf{y}$ agree in a coordinate where both values are $1$.
Adding a vector cannot create a common fixed coordinate.  Hence non-triviality is also preserved, contradicting maximality.  Therefore
\[
   \calF=\{\mathbf{x}:P(\mathbf{x})\in\calU\},
\]
and consequently
\begin{equation}\label{can-calF}
   |\calF|=\sum_{A\in\calU}\prod_{i\notin A}(n_i-1).
\end{equation}

Set $\mathcal{D}:=\{[r]\setminus A:A\in\calU\}$. We claim that $\mathcal{D}$ is available. The condition $\bigcap\calU=\emptyset$ gives $\bigcup \mathcal{D}=[r]$. Since $\calU$ is intersecting, no two sets of $\mathcal{D}$ have union $[r]$.  It suffices to show $\mathcal{D}$ is a downset. Otherwise, there would exist $[r]\setminus A\in \mathcal{D}$ and $[r]\setminus B\subseteq [r]\setminus A$ such that $[r]\setminus B\notin \mathcal{D}$. This means that $A\subseteq B$, $A\in\calU$ and $B\notin\calU$. 
Choose $\mathbf{x}\in\calF$ with $P(\mathbf{x})=A$. Since $\calF$ is coordinate-wise shifted, replacing any coordinate of $\mathbf{x}$ outside $A$ by $1$ keeps the resulting vector in $\calF$; iterating over the coordinates in $B\setminus A$ gives a member of $\calF$ whose projection is $B$, a contradiction.

Recall that for $X\in [r]$, $a_X=\prod_{i\in X}(n_i-1)$. By \eqref{can-calF}, we have  that 
\[
|\calF|\le \max\limits_{\mathcal{D}\in\mathfrak D_r}\sum\limits_{X\in \mathcal{D}}a_X.
\]
Conversely, let $\mathcal{D}\in\mathfrak D_r$ and put
\[
   \calU:=\{[r]\setminus X:X\in \mathcal{D}\}.
\]
Then $\bigcap \calU=\emptyset$, and $\calU$ is intersecting because
$X\cup Y\ne[r]$ for all $X,Y\in \mathcal{D}$.  
Define
\[
   \calF_\mathcal{D}:=\{\mathbf{x}\in[n_1]\times\cdots\times[n_r]:P(\mathbf{x})\in\calU\}.
\]
Then $\calF_\mathcal{D}$ is intersecting, since any two projection sets in $\calU$ intersect and
the corresponding vectors agree with value $1$ in a common coordinate.  
It remains to check that $\calF_\mathcal{D}$ has no common vertex.  
Since $\mathcal{D}$ is a downset, we have $\emptyset\in \mathcal{D}$, and hence $[r]\in\calU$.  
Thus $\calF_\mathcal{D}$ contains the all-one vector, so no vertex with value different from $1$ can belong to every member of $\calF_\mathcal{D}$.  
On the other hand, since $\bigcup_{X\in \mathcal{D}}X=[r]$, for each $i\in[r]$ there exists $X\in \mathcal{D}$ with $i\in X$.  
Then $A=[r]\setminus X$ belongs to $\calU$ and does not contain $i$.  
Hence $\calF_\mathcal{D}$ contains a vector whose $i$th coordinate is not equal to $1$.  
Therefore the vertex $1$ of $X_i$ is not common to all members of $\calF_\mathcal{D}$ either.  
Since this holds for all $i\in [r]$, the family $\calF_\mathcal{D}$ has no common vertex.
Hence $\calF_\mathcal{D}$ is non-trivial and intersecting,  and
\[
   |\calF_\mathcal{D}|=\sum_{X\in \mathcal{D}}a_X.
\]
Taking the maximum over $\mathcal{D}$ completes the proof of Theorem \ref{thm:downset}.
\end{proof}


Next, we will give the exact values of $m_0(1,n_1,\ldots,n_r)$ for $r=4,5,6$. 
Before that, we define an \emph{intersecting graph} as a graph where every pair of edges shares a common vertex.

{
\renewcommand{\thesubsection}{\Roman{subsection}}
\setcounter{subsection}{0}

\subsection{The case $r=4$}

Let $\calI_4$ denote the set of all intersecting graphs on vertex set $[4]$, including the empty graph.

\begin{theorem}\label{thm:r4}
For $n_i\ge2$ and $a_i=n_i-1$ for $i\in [4]$, we have 
\[
   m_0(1,n_1,n_2,n_3,n_4)
   =
   1+\sum_{i=1}^4a_i+
   \max_{G\in\calI_4}\sum_{ij\in E(G)}a_i a_j.
\]
In particular, if $n_1\ge n_2\ge n_3\ge n_4$, then
\[
   m_0(1,n_1,n_2,n_3,n_4)
   =
   \max\left\{
      n_1(n_2+n_3+n_4-2),
      \,
      n_1n_2n_3-(n_1-1)(n_2-1)(n_3-1)+n_4-1
   \right\}.
\]
Equivalently, under this ordering of the part sizes, the final maximum is attained either by the star centered at $1$ or by the triangle on $\{1,2,3\}$.

\end{theorem}

\begin{proof}
By Theorem \ref{thm:downset}, it suffices to optimize over $\mathcal{D}\in\mathfrak D_4$.
Since $\mathcal{D}$ is a downset and $\bigcup \mathcal{D}=[4]$, all singletons belong to $\mathcal{D}$.  Hence $\mathcal{D}$ cannot contain a $3$-set, because such a set together with the missing singleton would have union $[4]$.  
Also $[4]\notin \mathcal{D}$.  Thus $\mathcal{D}$ has only levels $0,1,2$.
Define a graph G on $[4]$ by
\[
   E(G):=\mathcal{D}\cap\binom{[4]}2.
\]
If $e,f\in E$ are disjoint, then $e\cup f=[4]$, contradicting the definition of $\mathcal{D}$. Hence $G\in\calI_4$.
Hence $E$ is an intersecting graph.  
Conversely, every intersecting graph $E$ yields an available family 
\[
   \mathcal{D}=\{\emptyset\}\cup\binom{[4]}1\cup E.
\]
Therefore
\[
   m_0(1,n_1,n_2,n_3,n_4)
   =1+\sum_{i=1}^4a_i+
    \max_{G\in\calI_4}\sum_{ij\in E(G)}a_i a_j.
\]

Assume now $n_1\ge n_2\ge n_3\ge n_4$.  
An intersecting graph on four vertices is contained in a star or in a triangle; since all edge weights are positive, an optimum is a full star or a full
triangle.  The largest star is centered at vertex $1$, giving contribution
\[
   a_1a_2+a_1a_3+a_1a_4,
\]
while the largest triangle is on $\{1,2,3\}$, giving contribution
\[
   a_1a_2+a_1a_3+a_2a_3.
\]
Adding $1+\sum_i a_i$ gives respectively
\[
   (1+a_1)(1+a_2+a_3+a_4)=n_1(n_2+n_3+n_4-2)
\]
and
\[
   n_1n_2n_3-a_1a_2a_3+a_4
   =n_1n_2n_3-(n_1-1)(n_2-1)(n_3-1)+n_4-1.
\]
The result follows.
\end{proof}

\subsection{The case $r=5$}

Let $\calI_5$ be the set of all intersecting graphs on $[5]$, including the empty graph.
For $e\in\binom{[5]}2$, write $\overline e:=[5]\setminus e$ and
\[
   \Delta_e:=a_{\overline e}-a_e.
\]

\begin{theorem} 
\label{thm:r5}
For $n_i\ge2$ and $a_i=n_i-1$ for $i\in [5]$, we have 
\[
   m_0(1,n_1,\ldots,n_5)
   =
   1+\sum_{i=1}^5a_i+
   \sum_{1\le i<j\le5}a_i a_j
   +
   \max_{G\in\calI_5}\sum_{e\in E(G)}\Delta_e.
\]
Equivalently, the final maximum may be taken over subgraphs of stars and triangles on $[5]$.
\end{theorem}

\begin{proof}
Let $\mathcal{D}\in\mathfrak D_5$ be optimal. As before, since $\mathcal{D}$ is a downset and
$\bigcup_{X\in \mathcal{D}}X=[5]$, all singletons belong to $\mathcal{D}$. Hence $\mathcal{D}$ contains no
$4$-set and no $5$-set. Thus only levels $0,1,2,3$ can occur.

For $e\in\binom{[5]}2$, we define a graph $G$ on $[5]$ by
\[
    E(G):=\left\{e\in\binom{[5]}2:\bar e\in \mathcal{D}\right\}.
\]
If $e,f\in E(G)$ were disjoint, then
\[
    \bar e\cup \bar f=[5],
\]
contradicting the defining condition of $\mathfrak D_5$. Hence $G\in\mathcal I_5$.

For every $e\in\binom{[5]}2$, the condition $X\cup Y\ne[5]$ forbids both $e$ and
$\bar e$ from belonging to $\mathcal{D}$. Conversely, since the original weights $a_X$ are
positive, optimality forbids both $e$ and $\bar e$ from being absent. Indeed, if
neither $e$ nor $\bar e$ belongs to $\mathcal{D}$, then adding the $2$-set $e$ preserves the
downset property, since all singletons already belong to $\mathcal{D}$. It also preserves the
condition $X\cup Y\ne[5]$, because the only set of size at most $3$ whose union with
$e$ is $[5]$ is $\bar e$, which is absent. Thus exactly one of $e$ and $\bar e$
belongs to $\mathcal{D}$. Therefore
\[
   \mathcal{D}=
   \{\emptyset\}\cup\binom{[5]}1\cup
   \left(\binom{[5]}2\setminus E(G)\right)
   \cup\{\bar e:e\in E(G)\}.
\]
Consequently,
\[
   \sum_{X\in \mathcal{D}}a_X
   =
   1+\sum_{i=1}^5a_i+\sum_{e\in\binom{[5]}2}a_e
   +\sum_{e\in E(G)}(a_{\bar e}-a_e).
\]
This gives the upper bound
\[
m_0(1,n_1,\ldots,n_5)
\le
1+\sum_{i=1}^5a_i+\sum_{1\le i<j\le5}a_ia_j
+\max_{G\in\mathcal I_5}\sum_{e\in E(G)}\Delta_e.
\]

Conversely, let $G\in\mathcal I_5$ and define
\[
\mathcal{D}_G=
\{\emptyset\}\cup\binom{[5]}1\cup
\left(\binom{[5]}2\setminus E(G)\right)
\cup\{\bar e:e\in E(G)\}.
\]
We claim that $\mathcal{D}_G\in\mathfrak D_5$. First, $\mathcal{D}_G$ is a downset. Since $\mathcal{D}_G$
contains $\emptyset$ and all singletons, the only nontrivial point is to check
the $2$-subsets of the $3$-sets $\bar e$ with $e\in E(G)$. Let
$f\in\binom{[5]}2$ and suppose $f\subseteq\bar e$ for some $e\in E(G)$. Then
$f\cap e=\emptyset$. Since $G$ is pairwise intersecting, we have
$f\notin E(G)$, and hence
\[
    f\in \binom{[5]}2\setminus E(G)\subseteq \mathcal{D}_G.
\]
Thus $\mathcal{D}_G$ is a downset.

It remains to verify the admissibility condition. Since $\mathcal{D}_G$ contains all
singletons, we have
\[
    \bigcup_{X\in \mathcal{D}_G}X=[5].
\]
Moreover, no two members of $\mathcal{D}_G$ have union $[5]$. Indeed, among sets of sizes at
most $3$, a full union can occur only from a complementary $2$--$3$ pair or from
two $3$-sets. A complementary $2$--$3$ pair cannot occur, because whenever
$\bar e\in \mathcal{D}_G$, the corresponding $2$-set $e$ has been deleted. If two $3$-sets
$\bar e,\bar f\in \mathcal{D}_G$ had union $[5]$, then $e\cap f=\emptyset$, contradicting
the fact that $G$ is pairwise intersecting. Hence $\mathcal{D}_G\in\mathfrak D_5$.

Therefore the upper bound is sharp, and
\[
   m_0(1,n_1,\ldots,n_5)
   =
   1+\sum_{i=1}^5a_i+\sum_{1\le i<j\le5}a_ia_j
   +\max_{G\in\mathcal I_5}\sum_{e\in E(G)}\Delta_e.
\]

Finally, any intersecting graph on five vertices is contained in a star or in a triangle.
The maximum $\max_{G\in\mathcal I_5}\sum_{e\in E(G)}\Delta_e$
may be checked over subgraphs of stars and triangles. 
Since the weights $\Delta_e$ may be negative, an optimal graph need not be a full star or a full triangle.
\end{proof}

\subsection{The case $r=6$}

Let $\calI_6$ be the set of all intersecting graphs on $[6]$, including the empty graph.
For $e\in\binom{[6]}2$, write $\overline e=[6]\setminus e$ and
\[
   \Delta_e:=a_{\overline e}-a_e.
\]
Set
\[
   B_2^{(6)}:=1+\sum_{i=1}^6a_i+
   \sum_{1\le i<j\le6}a_i a_j.
\]
For $G\in\calI_6$, define
\[
   \calA(G):=
   \left\{C\in\binom{[6]}3:\text{there is no }e\in E(G)\text{ with }e\subseteq C\right\}.
\]
The $3$-subsets of $[6]$ are paired by complementation. Define
\[
   \Phi(G):=
   \sum_{\{C,[6]\setminus C\}}
   \max\left(
      \{0\}\cup
      \{a_C:C\in\calA(G)\}\cup
      \{a_{[6]\setminus C}:[6]\setminus C\in\calA(G)\}
   \right),
\]
where each complementary pair is counted once.

\begin{theorem}
\label{thm:r6}
For $n_i\ge2$ and $a_i=n_i-1$ for $i\in [6]$, we have 
\[
   m_0(1,n_1,\ldots,n_6)
   =
   B_2^{(6)}+
   \max_{G\in\calI_6}
   \left(
      \sum_{e\in E(G)}\Delta_e+
      \Phi(G)
   \right).
\]
\end{theorem}

\begin{proof}
Let $\mathcal{D}\in\mathfrak D_6$ be optimal.  Since $\mathcal{D}$ contains all singletons, it contains no $5$-set and no
$6$-set.  Thus every member of $\mathcal{D}$ has size at most $4$.

Define
\[
   E(G):=\{e\in\binom{[6]}2:\overline e\in \mathcal{D}\}.
\]
If $e,f\in E(G)$ are disjoint, then $\overline e\cup\overline f=[6]$, impossible.  Thus
$G\in\calI_6$.  As in the $r=5$ case, optimality forces
\[
   \mathcal{D}\cap\binom{[6]}2=\binom{[6]}2\setminus E(G),
\]
for the only obstruction to adding a missing $2$-set $e$ is the presence of its complementary $4$-set $\overline e$.
Thus
\[
  \mathcal{D}\cap\binom{[6]}4=\{\overline e:e\in E(G)\}.
\]
The total contribution from levels \(0,1,2,4\) is
\[
    1+\sum_{i=1}^6a_i
    +\sum_{e\in\binom{[6]}2\setminus E(G)}a_e
    +\sum_{e\in E(G)}a_{\overline e}
    =
    B_2^{(6)}+\sum_{e\in E(G)}(a_{\overline e}-a_e).
\]
It remains to bound the possible contribution from the third level. Put 
\[
H:=\mathcal{D}\cap\binom{[6]}3.
\] 
We first show that $H\subseteq\calA(G)$, if $C\in H$ and some $e\in E(G)$ satisfies $e\subseteq C$, then $C\cup\overline e=[6]$,
contradicting the conditions of $\mathfrak{D}_6$.  
Hence $H\subseteq\calA(G)$. Also $H$ cannot contain both $C$ and $[6]\setminus C$. Hence the total weight of $H$ is at most $\Phi(G)$. 

Since \(\Delta_e=a_{\overline e}-a_e\), we obtain
\[
    \sum_{X\in \mathcal{D}}a_X
    \le
    B_2^{(6)}+\sum_{e\in E(G)}\Delta_e+\Phi(G).
\]
Taking the maximum over \(G\in\mathcal I_6\) gives the upper bound
\[
    m_0(1,n_1,\ldots,n_6)
    \le
    B_2^{(6)}
    +
    \max_{G\in\mathcal I_6}
    \left(
        \sum_{e\in E(G)}\Delta_e+\Phi(G)
    \right).
\]

We now prove the reverse inequality. Fix \(G\in\mathcal I_6\). For each
complementary pair \(\{C, [6]\setminus C\}\) of \(3\)-sets, choose a member of maximum
weight among the members of this pair that lie in \(\calA(G)\); if neither member lies
in \(\calA(G)\), choose nothing. Let \(H_G\) be the family of chosen triples. Then
\[
    H_G\subseteq \calA(G),
    \qquad
    \sum_{C\in H_G}a_C=\Phi(G).
\]
Define
\[
    \mathcal{D}_G:=
    \{\emptyset\}
    \cup\binom{[6]}1
    \cup\left(\binom{[6]}2\setminus E(G)\right)
    \cup H_G
    \cup\{\overline e:e\in E(G)\}.
\]
We claim that \(\mathcal{D}_G\in\mathfrak D_6\).
First, $\mathcal{D}_G$ is a downset. Since it contains $\emptyset$ and all singletons, the only nontrivial points to check are the chosen $3$-sets and the present $4$-sets.
Let $C\in H_G$. Since $C\in \calA(G)$, no edge of $G$ is contained in $C$.
Therefore every $2$-subset of $C$ lies in $\binom{[6]}2\setminus E(G)\subseteq \mathcal{D}_G$.

Now let $\overline{e}$ be a present $4$-set, with $e \in E(G)$. Clearly, all $2$-subsets of $\overline{e}$ belong to $\binom{[6]}2\setminus E(G)$.
Let $C \subseteq \overline{e}$ be any $3$-subset. 
Since $C \subseteq \overline{e} = [6]\setminus e$, every $2$-subset of $C$ is disjoint from $e$.
Thus, by the intersection of $G$, we must have $C\in\mathcal{A}(G)$. 
On the other hand, the complementary $3$-set $[6]\setminus C$ contains $e$, and therefore $[6]\setminus C \notin \mathcal{A}(G)$, the construction of $H_G$ forces $C\in H_G$. 
This proves that $\mathcal{D}_G$ is a downset.

Second, since $\mathcal{D}_G$ has no member of size larger than $4$, the only possible ways for two members to have union $[6]$ are of types $2-4$, $3-3$, $3-4$, and $4-4$.
For a $2-4$ pair, every present $4$-set has the form $\overline{e}$ with
$e\in E(G)$. However, in the construction of $\mathcal{D}_G$, the set $e$ has already
been deleted. Hence no complementary $2-4$ pair occurs.

For a $3-3$ pair, two $3$-sets have union $[6]$ exactly when they are complementary. But $H_G$ contains at most one member from each complementary pair, so this cannot occur.

For a $3-4$ pair, let $C\in H_G$ and let $\overline{e}$ be a present $4$-set with $e\in E(G)$. If $C\cup\overline{e}=[6]$, then $e\subseteq C$. 
This contradicts $C\in \calA(G)$. Hence no $3-4$ pair has full union.

Finally, consider two present $4$-sets $\overline{e}$ and $\overline{f}$, where $e, f\in E(G)$. Their union is $[6]$ if and only if $e\cap f=\emptyset$, this contradicts the intersection of $G$.
Thus $X\cup Y\neq [6]$ for all $X, Y\in \mathcal{D}_G$.

Third, since $\mathcal{D}_G$ contains all singletons, we have 
\[
\bigcup_{X\in \mathcal{D}_G} X=[6].
\]
Consequently, $\mathcal{D}_G\in\mathfrak{D}_6$.

It remains to compute its weight. By construction,
\begin{align*}
\sum_{X\in \mathcal{D}_G} a_X
&=
1+\sum_{i=1}^{6} a_i
+\sum_{e\in \binom{[6]}{2}\setminus E(G)} a_e
+\sum_{C\in H_G} a_C
+\sum_{e\in E(G)} a_{\bar e}\\
&=B_2^{(6)}+\sum_{e\in E(G)}\Delta_e+\Phi(G).
\end{align*}

Since this construction works for every $G\in \mathcal I_6$, we obtain
\[
m_0(1,n_1,\ldots,n_6)
\ge B_2^{(6)}+\max_{G\in \mathcal I_6}\left(\sum_{e\in E(G)}\Delta_e+\Phi(G)\right).
\]

Together with the upper bound, this proves
\[
m_0(1,n_1,\ldots,n_6)=B_2^{(6)}+\max_{G\in \mathcal I_6}\left(\sum_{e\in E(G)}\Delta_e+\Phi(G)\right).
\]
The theorem follows.
\end{proof}

\section{Declaration on the use of AI}
The authors used generative AI tools to assist checking proofs, and improving exposition.

\end{document}